\documentclass[11pt]{article}
\usepackage{amssymb,amsthm,amsmath,mathrsfs}
\usepackage{prettyref,cite}
\usepackage[colorlinks,linkcolor=blue,citecolor=magenta]{hyperref}
\usepackage{graphicx}
\usepackage{algorithm,algpseudocode}
\usepackage{caption,subcaption}
\usepackage [a4paper, footskip=1cm, headheight = 16pt, top=3cm, bottom=3cm,  right=2.5cm,  left=2.5cm ]{geometry}
\usepackage{tikz}

\newrefformat{chp}{\chaptername~\ref{#1}}
\newrefformat{sec}{Section~\ref{#1}}
\newrefformat{sub}{Subsection~\ref{#1}}
\newrefformat{app}{Appendix~\ref{#1}}
\newrefformat{sub}{Subsection~\ref{#1}}
\newrefformat{thm}{Theorem~\ref{#1}}
\newrefformat{qst}{Question\ref{#1}}
\newrefformat{lem}{Lemma~\ref{#1}}
\newrefformat{alg}{Algorithm~\ref{#1}}
\newrefformat{cor}{Corollary~\ref{#1}}
\newrefformat{con}{Conjecture~\ref{#1}}
\newrefformat{dfn}{Definition~\ref{#1}}
\newrefformat{rem}{Remark~\ref{#1}}
\newrefformat{fig}{\figurename~\ref{#1}}
\newrefformat{tbl}{\tablename~\ref{#1}}
\newrefformat{eq}{$($\ref{#1}$)$}
\newrefformat{fac}{Fact~\ref{#1}}

\newtheorem{pro}{Proposition}
\newtheorem{cor}[pro]{Corollary}
\newtheorem{lem}[pro]{Lemma}
\newtheorem{subclaim}{Claim}[pro]
\newtheorem{thm}[pro]{Theorem}

\newtheorem{qst}[pro]{Question}
\newtheorem{prerem}[pro]{Remark}
\newtheorem{predefin}[pro]{Definition}

\newenvironment{dfn}{\begin{predefin}\rm}{\hfill $\blacktriangle$\end{predefin}}

\DeclareMathOperator{\tw}{tw}
\DeclareMathOperator{\cn}{c}

\begin{document}

\title{\textbf{The Game of Cops and Robber\\ on (Claw, Even-hole)-free Graphs 
}}
\author{
		Ramin Javadi%
		\thanks{Corresponding author, Department of Mathematical Sciences,
			Isfahan University of Technology,
			P.O. Box: 84156-83111, Isfahan, Iran. School of Mathematics, Institute for Research in Fundamental Sciences (IPM), P.O. Box: 19395-5746,
			Tehran, Iran.  Email Address: \href{mailto:rjavadi@iut.ac.ir}{rjavadi@iut.ac.ir}.}
		\and
		Ali Momeni%
		\thanks{
		Department of Mathematical Sciences, Isfahan University of Technology, P.O. Box: 84156-83111, Isfahan, Iran. Email Address: \href{mailto:alimomeni.ma@gmail.com}{alimomeni.ma@gmail.com}.
		} 		 
	}

\date{}
\maketitle
\begin{abstract}
In this paper, we study the game of cops and robber on the class of graphs with no even hole (induced cycle of even length) and claw (a star with three leaves). 
The cop number of a graph $G$ is defined as the minimum number of cops needed to capture the robber. 
Here, we prove that the cop number of all claw-free even-hole-free graphs is at most two and, in addition, the capture time is at most $2n$ rounds, where $n$ is the number of vertices of the graph. 
Moreover, our results can be viewed as a first step towards studying the structure of claw-free even-hole-free graphs.

\end{abstract}
\section{Introduction}
The game of cops and robber is a two-player turn-based game between the first player who controls a fixed number of cops and the second player who controls a robber.
The game is played on a given connected graph $G=(V, E)$ on which the cops pursue the robber to capture him, and the robber tries to escape.
Each \textit{round} of the game consists of a cops' turn followed by a
robber's turn.
In the first round, the first player locates the cops on some vertices of $G$ and then the second player locates the robber on a vertex. 
Afterward, the cops and the robber move consecutively; at each cops' turn (resp. robber's turn), every single cop (resp. the robber) can stay in his position or move to an adjacent vertex. 
Once a cop is in the robber's position, the game is over, and the first player wins. 
The \textit{cop number} of a connected graph $G$, denoted by $\cn(G)$, is the minimum number of cops which are sufficient to capture the robber on $G$, i.e., the minimum number $k$ such that $k$ cops have a winning strategy on $G$. 
We follow \cite{BERARDUCCI1993389,Joret:2008aa} to define the cop number of a disconnected graph as the maximum of the cop numbers of its connected components. Also, $G$ is called to be \textit{$k$-cop-win}, if $\cn(G)= k$ and is called \textit{cop-win}, if $\cn(G)=1$. 
For instance, Petersen graph is the smallest $3$-cop-win graph \cite{Beveridge:2011aa} and a $4$-cop-win graph has at least $19$ vertices \cite{turcotte20214}. 
For other variants of the game, see e.g. \cite{doi:10.1002/jgt.20591,Luckraz_2018,gavenvciak2016structural,FOMIN2008236}.

For a $k$-cop-win graph $G$, the \textit{capture time} of $G$ is the minimum number $t$ such that $k$ cops can always capture a robber on $G$ within $t$ rounds, regardless of how the robber plays.
It can be easily seen that for every positive integer $k$, the capture time of an $n$-vertex $k$-cop-win graph is at most $O(n^{k+1})$.
Bonato et al. \cite{BONATO20095588} showed that this bound is not tight for $k=1$, proving that the capture time of all cop-win graphs on $n$ vertices is at most $n-3$. 
However, surprisingly, the upper bound happens to be asymptotically tight for all integers $k\geq 2$, i.e. 
there are $k$-cop-win graphs with the capture time in $\Theta(n^{k+1})$ for every integer $k\geq 2$ \cite{kinnersley2018bounds,BRANDT2020143}.

 The problem of finding the exact value or bounds on the cop numbers of graphs has been studied extensively for decades.
 For a class of graphs $\mathcal{H}$, we say that $\mathcal{H}$ is \textit{cop-bounded} if there is a constant $k= k(\mathcal{H})$, such that for every graph $H\in \mathcal{H}$, $c(H)\leq k$. 
For instance, it is easy to see that every tree is cop-win, so the class of all trees is cop-bounded.
It is a classic and well-studied question in the literature: which classes of graphs are cop-bounded?
One of the important classes that this question is studied is the class of graphs that exclude a (some) certain graph(s) as a subgraph, induced subgraph, or minor. 
Given a graph $H$, we say that $G$ is $H$-free, $H$-subgraph-free, or $H$-minor-free if $G$ does not contain an induced subgraph, subgraph, or minor isomorphic to $H$, respectively. 
Also, for a class of graphs $\mathcal{H}$, we say that $G$ is $\mathcal{H}$-free if $G$ is $H$-free for every $H\in \mathcal{H}$ (the same definition can be stated for $\mathcal{H}$-subgraph-free and $\mathcal{H}$-minor-free).
The first results in the context date back to 1984 when it was proved in  \cite{AIGNER19841}  that all planar graphs have the cop number at most $3$ and in \cite{ANDREAE1984111} that the class of $k$-regular graphs are cop-unbounded for every integer $k\geq 3$. 
The latter result can be also implied from a recent result in \cite{Bradshaw:2020aa} stating that the cop number of any graph with girth \( g \) and minimum degree \( \delta \) is at least \( \frac{1}{g} \left( \delta - 1 \right) ^{ \lfloor \frac{g-1}{4} \rfloor } \).
 Also, there is a generalization of the former result stating that for every graph $G$, \( \cn(G) \leq (4 g(G)+10)/{3}  \), where $g(G)$ is the genus of the graph $G$ \cite{doi:10.1137/20M1312150}.
   In addition, Andreae in \cite{ANDREAE198637} proved that for every fixed graph $H$, the class of $H$-minor-free graphs is cop-bounded. 
More precisely, if $h $ is a vertex of \( H \) such that \( H-h \) has no isolated vertex, then the cop number of every $H$-minor-free graph is at most $ |E( H - h)|$.

 Joret et al. proved in \cite{Joret:2008aa} that for every graph $H$, the class of $H$-free graphs is cop-bounded if and only if every connected component of $H$ is a path. 
This result implies 
that the class of claw-free graphs are not cop-bounded. 
In particular, 
if $G$ is \( P_t \)-free, then \( t -2 \) cops can capture the robber in at most \( t -1 \) moves \cite{sivaraman2019application}.

For a class of graphs $\mathcal{H}$, it is proved in \cite{masjoody2020cops} that 
if every connected component of the graphs in \( \mathcal{H} \) has a bounded diameter, then the class of $\mathcal{H}$-free graphs is cop-bounded if and only if \( \mathcal{H} \) contains either a forest of paths or two graphs $F_1$, $F_2$, each having at least one vertex of degree three
such that every connected component of $F_1$ is a path or a generalized claw,
and every connected component of $F_2$ is a path or a generalized net.
However, the problem of characterization of all classes $\mathcal{H}$ for which $\mathcal{H}$-free graphs are cop-bounded is still open.
One interesting case is when $\mathcal{H}$ is a class of some holes with unbounded diameters. 
By a hole, we mean a cycle of length at least $4$ without any chord.
For instance, it is proved in \cite{Sivaraman:2020aa} that if $\mathcal{H}$ is the class of all holes of length at least \( k\),  \( (k \geq 4) \), then the cop-number of every $\mathcal{H}$-free graph is at most \( k-3 \).
It is also known that any class of graphs with bounded treewidth is cop-bounded. 
In fact, Joret et al. \cite{Joret:2008aa} proved that for every graph $G$, $\cn(G)\leq 1/2 \tw(G)+1$, where $\tw(G)$ is the treewidth of the graph $G$.

 Meanwhile, the structure of odd-hole-free and even-hole-free graphs have been studied widely in the literature, mostly motivated by the study and generalization of perfect graphs  (for instance, see \cite{SCOTT201668,CONFORTI200441,CHUDNOVSKY2010313,DASILVA2013144}). 
This gives rise to the following natural question that whether the classes of odd-hole-free and even-hole-free graphs are cop-bounded. 
At first glance, one may easily check that odd-hole-free graphs are cop-unbounded since subdividing all edges of a graph does not decrease its cop number  \cite{Joret:2008aa,BERARDUCCI1993389}. 
How about when we also exclude claws? One may see that  (claw, odd-hole)-free graphs are also cop-unbounded. 
To see this, we need a notion called \textit{clique substitution} defined in \cite{Joret:2008aa}. 
Given a graph $G$, its clique substitution is the graph $H$ obtained from $G$ by replacing each vertex $v$ of $G$ by a clique $C_v$ of the size $\deg(v)$ such that, if $u$ is adjacent to $v$ in $G$, then a vertex in $C_v$ is adjacent to a vertex in $C_u$ in a way that every vertex in a clique $C_v$  has exactly one neighbor outside $C_v$. 
Joret et al. \cite{Joret:2008aa} proved that if $H$ is the clique substitution of $G$, then $\cn(G)\leq \cn(H)$. 
Also, it is clear that whatever $G$ is, $H$ is always odd-hole-free and claw-free as well. 
Hence, if $\mathcal{H}$ is a cop-unbounded class, then the class of clique substitutions of the graphs in $\mathcal{H}$ is also cop-unbounded, and they are evidently (claw, odd-hole)-free. 
So, we have
\begin{cor}
The class of all (claw, odd-hole)-free graphs is cop-unbounded.
\end{cor}

It is not so straightforward to answer the same question for even-hole-free graphs.
It is proved that the treewidth of (triangle, even-hole)-free graphs is at most $5$ \cite{CAMERON2018463}, so the cop number of (triangle, even-hole)-free graphs is at most $3$. 
On the other hand, (claw, even-hole)-free graphs contain all cliques, so the class has unbounded treewidth. 
In this paper, we prove, in contrast, that the class of (claw, even-hole)-free graphs is cop-bounded.

\begin{thm} \label{thm:main}
If $G$ is a (claw, even-hole)-free graph with $n$ vertices, then $\cn(G)\leq 2$ and the capture time of $G$ is at most $2n$. 
\end{thm}

Our proof also gives an insight into the structure of such graphs.
It remains open to answer the following question in general.

\begin{qst}\label{qst:main}
Is it true that the class of even-hole-free graphs is cop-bounded?
\end{qst}

Recently, there have been many pieces of work on the study of even-hole-free graphs. 
Here, we survey some most important results in this area.
A decomposition theorem for even-hole-free graphs is provided in \cite{DASILVA2013144}
strengthening a previously known decomposition result in \cite{conforti2002even}.
For a survey on the structural results, see \cite{vuvskovic2010even}.
 Chudnovsky and Seymour in \cite{Chudnovsky:2019tg} settled a conjecture by Reed and proved that every even-hole-free graph has a bisimplicial vertex. 
A bisimplicial vertex is a vertex whose neighborhood can be partitioned into two cliques.
Also, Chudnovsky et al. \cite{chudnovsky2021note} motivated by an application in condensed matter physics and quantum information theory, proved that every non-null (claw, even-hole)-free graph has a simplicial clique. 
By a simplicial clique, we mean a clique $C$ such that for every vertex $v$ in $C$, the neighbors of $v$ outside $C$ form a clique.

The treewidth of even-hole-free graphs has also been studied extensively. 
In particular, Cameron et al. \cite{CAMERON2018463} providing a structure theorem for (cap, even-hole)-free graphs proved that they are \( \chi \)-bounded, i.e., their chromatic number is bounded by their clique number. 
This raises the question of whether the treewidth of an even-hole-free graph can be bounded by its clique number? This question was answered negatively by Sintiari and Trotignon \cite{Sintiari:2019aa} by constructing a family of (\( K_4 \), even-hole)-free called \textit{layered wheels}, whose treewidth is unbounded.
Moreover, it is proved in \cite{Aboulker:2020aa} that every class of even-hole-free graphs that excluding a fixed graph as a minor has bounded treewidth and  Abrishami et al. \cite{Abrishami:2020ab}, resolving a conjecture given in \cite{Aboulker:2020aa},  proved that 
even-hole-free graphs of bounded degree have bounded treewidth.

The structure of forthcoming sections is as follows. In Section~\ref{sec:structure}, we investigate the structure of (claw, even-hole)-free graphs inside and between its distance levels. In Section~\ref{sec:holes}, we study the structure of holes in (claw, even-hole)-free graphs and their setting with respect to the distance levels. In Section~\ref{sec:algorithm}, we give an algorithm for the game of cops and robber on (claw, even-hole)-free graphs which guarantees that the cop number of these graphs is at most three. Finally, in Section~\ref{sec:refinement}, we improve the algorithm to prove Theorem~\ref{thm:main}.

\subsection{Notations and Conventions} \label{sub:notation}
Given a graph $G=(V,E)$ and a vertex $v\in V$, the set of neighbors of $v$ in $G$ is denoted by $N_G(v)$ and, whenever there is no ambiguity, we drop the subscript. 
For a subset of vertices $A\subseteq V$, the set of neighbors of $v$ in $A$ is denoted by $N_A(v)$, i.e. $N_A(v)=A\cap N_G(v)$. 
The induced subgraph of $G$ on $A$ is denoted by $G[A]$ which is a graph with vertex set $A$ and all connections as in the graph $G$. 
The graph obtained from $G$ by removing all vertices in $A$ is denoted by $G-A$. Also, for a pair of non-adjacent vertices $x,y\in V$, the graph $G+xy$ is obtained from $G$ by adding the edge $xy$ to $G$.
For two disjoint subsets of vertices $A, B\subset V$, we say that $A$ is complete (resp. incomplete) to $B$ if every vertex in $A$ is adjacent (resp. non-adjacent) to every vertex in $B$. 
Also, we say that $A$ is connected to $B$ if $A$ is not incomplete to $B$.

A sequence of distinct vertices $P_k:u_1-u_2-\cdots-u_k$ is called a path  of length $k-1$ if for every $i\in\{1,\ldots,k-1\}$, $u_i$ is adjacent to $u_{i+1}$. 
Sometimes we consider a direction on $P$ and say that $P$ is a path from $u_1$ to $u_k$. 
So, the vertices $u_{i-1}$ and $u_{i+1}$ are defined as the last vertex before and the first vertex after $u_i$, respectively. 
Also, for some indices $i,j$, $i<j$, the subpath $u_i-u_{i+1}-\cdots-u_j$ is denoted by $u_i-P-u_j$.

Let $u_0$ and $u_1$ be two arbitrary adjacent vertices in $V$. 
Also, let $B_0=N(u_0)\setminus \{u_1\}$. 
Now, we define $G_0 = G_0(G,u_0,u_1)$ to be the connected component of $G- B_0$ containing $u_0$. 
For each positive integer $j$, define the level $L_j$ as the set of vertices of distance $j$ from $u_0$ in the graph $G_0$, i.e. $L_j=\{u\in V(G_0): d_{G_0}(u,u_0)=j \}$.
For simplicity, we use the notations $L_{\geq i}=\cup_{j\geq i} L_j$ and $L_{> i}=\cup_{j> i} L_j$.
For every vertex $v\in L_i$, the set of neighbors of $v$ in $L_{i+1}$ is denoted by $N^+(v)=N_{L_{i+1}}(v)$ and 
the set of neighbors of $v$ in $L_{i-1}$ is denoted by $N^-(v)=N_{L_{i-1}}(v)$.
Every connected component of the induced subgraph of $G_0$ on $L_i$, i.e. $G_0[L_i]$, is called a \textit{level component} in $L_i$.

\section{The Structure of Level Components}\label{sec:structure}

In this section, we study the structure of a (claw, even-hole)-free graph. 
Let $G$ be a (claw, even-hole)-free graph and $G_0$ and the levels $L_j$'s be defined as in Subsection \ref{sub:notation}.
Here, we investigate the form of the level components in each level and how they are connected to other levels of the graph $G_0$.

Since $G$ is claw-free, it is an easy observation that, for every vertex $u\in L_i$, $i \in \mathbb{Z}^+$, the set of neighbors of $u$ in the next level $L_{i+1}$, \( N^+ (u)\), is a clique. 
It is because $u$ has a neighbor in $L_{i-1}$ and if it has two non-adjacent neighbors in $L_{i+1}$, then there is a claw with the center $u$. 
Thus,
\begin{lem} \label{lem:clique}
	For every vertex \( u \in L_i \), \( N^+ (u)\) is a clique.
\end{lem}
Now, we define a kind of path which is forbidden in $G_0$.
\begin{dfn}
Let \( P = x_1 - x_2 - \cdots - x_k \) be a path of odd length such that \( x_1, x_k \in L_i \), for some $i$, and \( x_2, x_{k-1} \in L_{i+1} \), and all other vertices are in $L_{\geq i}$. 
We say that $P$ is a \textit{forbidden path}, if either \( P \) is an induced path or \( P + x_1x_k \) is a hole. 
\end{dfn}
In the following lemma, we prove that there is no forbidden path in $G_0$.
\begin{lem} \label{lem:forbidden_path}
	There is no forbidden path in the graph $G_0$.
\end{lem}

\begin{proof}
	Suppose that \(P= x_1 - x_2 - \cdots - x_k \) is a forbidden path such that \( x_1, x_k \in L_i \), for some $i$, and \( x_2, x_{k-1} \in L_{i+1} \), and all other vertices are in $L_{\geq i}$.
	
	Let \( P^{x_1}, P^{x_k} \) be two shortest paths from \( x_1 \) to \( u_1 \) and from \( x_k \) to \( u_1 \), respectively.
	Then, there is an edge with one endpoint in \( V(P^{x_1}) \) and another in \( V(P^{x_k}) \), because if the vertex $u\in L_j$ is the first common vertex of $P^{x_1}$ and $P^{x_k}$, then, by Lemma~\ref{lem:clique}, the neighbors of $u$ in $L_{j+1}$ form a clique and thus, the neighbors of $u$ in $P^{x_1}$ and $P^{x_1}$ are adjacent. 
Now, let \( x_1'x_k' \) be the first edge between \( P^{x_1} \) and \( P^{x_k} \) in the sense that  $x'_1$ is the first vertex in $P^{x_1}$ (by passing from \( x_1 \) to \( u_1 \)) which has a neighbor in $P^{x_k}$ and $x'_k \in V(P^{x_k})$ is the closest neighbor of $x_1'$ to $x_k$.
	Both vertices \( x_1'\) and \(x_k' \) are located in the same level, since otherwise, by Lemma~\ref{lem:clique}, there is another edge between \( P^{x_1} \) and \( P^{x_k} \) which is a contradiction with the choice of $x'_{1}$ and $x_k'$. 
Note that if $x_1$ is adjacent to $x_k$, then $x_1=x_1'$ and $x_k=x_k'$.
	
	But now, \( x_1' - P^{x_1} - x_1 - P - x_k - P^{x_k} - x_k' - x_1' \) is an even hole.
	To see this, suppose that \( x \in P^{x_1} \) be adjacent to a vertex \( y \in P \setminus \{x_1 \} \).
	 Then, \( x \in L_{i-1} \) and, by Lemma \ref{lem:clique}, \( x_1\) is adjacent to  \( y \), contradicting the fact that \( P\setminus \{x_k \} \) is an induced path.
\end{proof}

\begin{cor} \label{col:xy_edge}
	For every edge \( xy \) in \( L_i \), either \( N^-(x) \subset N^-(y) \) or \( N^-(y) \subset N^-(x) \).
\end{cor}

\begin{proof}
	Suppose that there are two vertices \( x', y' \in L_{i-1} \) such that \( x' \in N^-(x) \setminus N^-(y) \) and \( y' \in N^-(y) \setminus N^-(x) \).
	Now, \( x' - x - y - y' \) is a forbidden path, contradicting Lemma~\ref{lem:forbidden_path}.
\end{proof}

In the next result, we determine the structure of $G_0$ inside each level $L_i$.
First, we need the following lemma from  \cite{wolk1965note}. 
Let $u$ be a vertex in a graph $G$ and $C$ be the connected component of $G$ containing $u$. The vertex $u$ is called a \textit{king vertex} in $G$, if $u$ is complete to $V(C)\setminus \{u\}$.

\begin{lem} {\rm \cite{wolk1965note}} \label{lem:king_vertex_general}
	Every connected \( (P_4,C_4) - \)free graph has a king vertex.
\end{lem}

By a \textit{$2$-clique}, we mean a connected graph $G = G(A,B,K)$ whose vertex set is partitioned into three nonempty sets $A,B,K$ such that $A \cup K$ and $B \cup K$ are cliques and $A$ is incomplete to $B$.

In the following lemma, we determine the structure of $G_0$ in each level.

\begin{lem} \label{lem:cc_form}
	For every integer $i \in \mathbb{Z}^+$, every level component in $L_i$ is either a clique or a $2$-clique.
\end{lem}

\begin{proof}
	First, we claim that the graph $G_0[L_i]$ is $P_4$-free.
	To see this, suppose that there is an induced path \( x-y-z-w \) in \( L_i \).
		Since \( x\) is non-adjacent to \( z \), by Lemma~\ref{lem:clique}, \( N^-(x) \cap N^-(z) = \emptyset \).
		Now, since \( y \) is adjacent to both   \( x\) and \( z \), by Corollary~\ref{col:xy_edge}, \( N^-(y) \supset N^-(x) \cup N^-(z) \).	
		But since \( y\) is non-adjacent to  \( w \), \( N^-(y) \cap N^-(w) = \emptyset \), which means \(  N^-(w) \cap N^-(z) = \emptyset \), contradicting Corollary~\ref{col:xy_edge}. 
Therefore, $G_0[L_i]$ is $(P_4,C_4)$-free, and thus, by Lemma~\ref{lem:king_vertex_general}, each of its level components has a king vertex.
	
	Let \( C\) be a level component of \( G_0[L_i] \) and \( K \subseteq V(C) \) be the set of king vertices of \( C \).
	If $K=V(C)$, then $C$ is a clique, and we are done. 
So, suppose that \( K \neq V(C) \), which means there are at least two non-adjacent vertices \( a_0,b_0 \in V(C) \setminus K \) that are adjacent to a vertex \( k \in K \).
	
	By Corollary~\ref{col:xy_edge}, there is a vertex \( u \in N^-(k) \cap N^-(b_0) \).
	By Lemma~\ref{lem:clique},  \( u\) is non-adjacent to \( a_0 \). 
Define \( A = \{x \in V(C) \mid x\text{ is non-adjacent to }u \} \) and \( B = V(C) \setminus (A \cup K) \).
 Since \(u\) is complete to \(B\), by  Lemma~\ref{lem:clique}, \( B \) is a clique.
	Also, \( A \) is a clique, since otherwise, there must be two non-adjacent vertices \( a, a' \in A \) and \( \{ k, a, a', u \} \) would induce a claw, a contradiction.

	Also, $A$ is incomplete to $B$. 
To see this, suppose that \( a \in A \) is adjacent to \( b \in B \). 
Since \( a,b \notin K \), there are vertices \( c \in A \) and \( d \in B \) such that \( b\) is non-adjacent to \( c \) and \( d\) is non-adjacent to \( a \).
		Then $G_0$ induces a \( P_4 \) or a \( C_4 \) on  \( \{ d, b, a, c \} \) which is in contradiction with the fact that $G_0[L_i]$ is $(P_4,C_4)$-free.
This proves that $C = C(A,B,K)$ is a $2$-clique.
\end{proof}

The next lemma determines the connections between two consecutive levels $L_{i-1}$ and $L_{i}$.

\begin{lem} \label{lem:backward_neighbours}
	Let \( C \) be a level component in $L_i$. 
Also, let $U\subseteq L_{i-1}$ be the set of all vertices in $L_{i-1}$ which has a neighbor in $C$.
	
	\begin{itemize}
		\item[\rm (i)] If $C=C(A,B,K)$ is a $2$-clique,	then for every \( u \in U \), we have either	\( N^+(u)= A \cup K\), or  \(N^+(u)= B \cup K\).
		\item[\rm (ii)]
			If \( C \) is a clique, then there is an ordering on the vertices in $U$, say $u_1,u_2,\ldots, u_\ell$, such that 
		\[N^+(u_1)\subseteq N^+(u_2)\subseteq \cdots \subseteq N^+(u_\ell)=C.\]	
	\end{itemize}
\end{lem}

\begin{proof}
	First, note that, by Lemma~\ref{lem:clique}, every $u\in U$ has no neighbor in other level components in $L_i$, i.e. $N^+(u)\subset C$.
\begin{enumerate}
	\item[(i)] Suppose that $C=C(A,B,K)$ is a $2$-clique. 
Note that  $u$ has a neighbor in $A\cup B$, since if $u$ is incomplete to $A\cup B$, then $u$ has a neighbor $k\in K$, and for every $a\in A$ and $b\in B$, $G_0$ induces a claw on $\{k,u,a,b\}$. 
So, without loss of generality, suppose that $u$ is adjacent to some vertex $a\in A$. 
Also, let $b$ be a vertex in $B$.
	Now, we prove that $K\subset N^+(u)$. 
To see this, suppose that $u$ is non-adjacent to a vertex $k' \in K$. 
Then, by Corollary~\ref{col:xy_edge}, $k'$ and $b$ has a common neighbor $v$ in $L_{i-1}$. 
Since $a$ is non-adjacent to $b$, by Lemma~\ref{lem:clique}, $a$ is also non-adjacent to $v$. 
Thus, $u-a-k'-v$ would be a forbidden path, contradicting Lemma~\ref{lem:forbidden_path}.
This proves that $u$ is complete to $K$.

	Now, we prove that \( u \) is complete to \( A \) and incomplete to \( B \). 
To see this, note that since \( A \) is incomplete to \( B \) and $u$ is adjacent to $a\in A$, by Lemma~\ref{lem:clique}, \( u \) is incomplete to \( B \).
		On the other hand, if $u$ is non-adjacent to some vertex $a'\in A$, then \( \{ k, a', b, u \} \) would be a claw.
This completes the proof of (i).

	\item[(ii)] Suppose that $C$ is a clique. 
We prove that for every $u,v\in U$, we have either $N^+(u)\subseteq N^+(v)$, or $N^+(v)\subseteq N^+(u)$.
For the contrary, assume that there are two vertices \( a, b \in C \) such that \( a \in N^+(u) \setminus N^+(v) \) and \( b \in N^+(v) \setminus N^+(u) \). 
Therefore,  \(u - a - b - v \) would be a forbidden path, contradicting Lemma~\ref{lem:forbidden_path}.
This proves that the inclusion induces a total ordering on $\{N^+(u), u\in U \}$.

Now, let \( u_\ell \in U \) be the vertex with the largest number of neighbors in $L_i$ among all vertices in $U$. 
If \( N^+(u_\ell) \neq C \), then \(u_\ell \) is non-adjacent to a vertex \( x \in C \).
The vertex $x$ has a neighbor $u\in U$, so  \( N^+(u_\ell) \subsetneq N^+(u) \), which is in contradiction with the choice of \( u_\ell \).
\end{enumerate}
\end{proof}

\section{Holes in (Claw, Even-hole)-free Graphs}\label{sec:holes}

In this section, we study the structure of holes in the graph $G_0$.
Let $H$ be a hole in $G_0$. 
The first and last levels $L_{i_0}$ and $L_{j_0}$ which intersect $V(H)$ is called the \textit{first and last level} of $H$, i.e. $i_0=\min\{i: L_i\cap V(H)\neq \emptyset\}$ and $j_0=\max\{i: L_i\cap V(H)\neq \emptyset\}$. 
Also, all levels $L_i$, $i_0<i<j_0$, are called \textit{inner levels} of $H$.
For every level $L_j$, $i_0\leq j\leq j_0$, a level component in $L_j$ that contains some vertices of $H$ is called a \textit{component} of $H$ in $L_j$.
Also, define
\[\tau(H)=\{C: C\text{ is a component of }H \text{ in $L_j$ for some } i_0+1 \leq j \leq j_0 \}. 
\] 

Note that, in the definition of $\tau(H)$, we have excluded the components of $H$ in its first level. 
In the next two lemmas, we prove that every hole $H$ in $G_0$ has a specific structure with respect to the levels $L_i$'s (see Figure~\ref{fig:hole_forms}).
\begin{lem} \label{lem:edge_in_levels_H}
Let \( H \) be a hole in the graph \( G_0 \).
The hole $H$,
	\begin{itemize}
		\item[\rm (i)] has no edge with both endpoints in a level $L_j$, for some $j \geq i_0 + 1$, and

		\item[\rm (ii)] has exactly one edge with both endpoints in its first level.
	\end{itemize}
\end{lem}
\begin{proof}
	\begin{itemize}
	\item[(i)]
Let \( i \) be the largest integer such that \( L_i \) contains an edge of \( H \), say \( xy \), and suppose that \( L_i \) is not the first level of \( H \).
	Orient the cycle \( H \) in a direction such that the edge $xy$ is directed from \(x\) to \(y\). 
Let $S=V(H)\cap L_{\geq i}$ and consider $H[S]$ the induced subgraph of $H$ on $S$. 
Now, define \( P \) as the connected component of $H[S]$ containing the edge $xy$ (note that $P$ is a subpath of $H$ which lies in $L_{\geq i}$).	

First, suppose that there is no edge in $H$ adjacent to $xy$ with both endpoints in $L_i$.  	
	Now, if $P$ does not start from $x$, then let \( x' \) be the last vertex of \( P \) in \( L_i \)  that is before \( x \) and let $x''$ be an arbitrary neighbor of $x'$ in $L_{i-1}$ and if \(P\) starts from \(x\), then define \(x'=x\) and $x''$ be the  neighbor of \( x' \) in \( L_{i-1}\cap V(H) \). Similarly define \( y' \) as the first vertex of \( P \) in \( L_i \) after \( y \) and $y''$ as an arbitrary neighbor of $y'$ in $L_{i-1}$ (if \(P\) ends in  \(y\), then define \(y'=y\) and \( y'' \) as the neighbor of \( y' \) in \( L_{i-1}\cap V(H) \)). 
Now, \(P' = x'' - x' - P - y' - y''\) is a forbidden path, contradicting Lemma~\ref{lem:forbidden_path}. 
To see this, note that $x''$ is incomplete to $V(P')\setminus \{x', y''\}$, because if $x'=x$, then $x''\in V(H)$ and if $x'\neq x$ and $x''$ is adjacent to a vertex in $u\in V(P')\setminus \{x', y''\}$, then $u$ is adjacent to $x'$ which is in contradiction with the fact that $H$ is a hole.
With a similar discussion, $y''$ is incomplete to $V(P')\setminus \{x'', y'\}$.

Finally, suppose that there is another edge of \( H \) adjacent to the edge \( xy \) in \( L_i \) and, without loss of generality, name this edge as \( yz \). Define $x'$ and $x''$ exactly as above and let $y'=y$ and $y''$ be a common neighbor of  \( y \) and \( z \) in $L_{i-1}$ which exists by Corollary~\ref{col:xy_edge}. 
Hence, again $P'= x'' - x' - P - y' - y''$ is a forbidden path, a contradiction.
	
	\item[(ii)]
	By (i), \( H \) has an odd number of edges in its first level \( L_{i_0} \).
	Suppose that \( H \) has more than one edge in its first level.
	Therefore, \( H \) has at least \( 3 \) edges in \( L_{i_0} \).
	So, \( H \) has a subpath of the form
	\[
	x' - x - P^x - y' - y - P^y - z' - z,
	\]
	where \( V(P^x), V(P^y) \subset L_{>i_0} \) and \( x', x, y', y, z', z \in L_{i_0} \).	
	
	The vertices \( z' \) and \(z \) have a common neighbor \(z''\) in \( L_{i_0-1} \) and, the vertices $x$ and $x'$ have a common neighbor $x''$ in $L_{i_0 - 1}$
	Now, the path \(P= x'' - x - P^x - y' - y - P^y - z' - z'' \) is a forbidden path, because $P\setminus \{x'',z''\}$ is a subpath of $H$ and so is induced. 
Also, due to Lemma~\ref{lem:clique}, $x''$ and $z''$ are both non-adjacent to inner vertices of $P\setminus \{x'',z''\}$.
	\end{itemize}
\end{proof}

\begin{figure}[t]
\centering

\begin{tikzpicture}

\draw[fill=black] (0,0) circle (2pt);
\draw[fill=black] (0,2) circle (2pt);
\draw[fill=black] (3,-.5) circle (2pt);
\draw[fill=black] (3,2.5) circle (2pt);

\draw[fill=black] (9,-.5) circle (2pt);

\draw[fill=black] (9,2.5) circle (2pt);

\draw[fill=black] (12,1) circle (2pt);



\node at (12,-2) {\LARGE \( L_{j_0} \)};
\node at (9,-2) {\LARGE \( L_{j_0-1} \)};

\node at (6,-2) {\LARGE \( \cdots \)};
\node at (6,-.5) {\LARGE \( \cdots \)};
\node at (6,2.5) {\LARGE \( \cdots \)};

\node at (3,-2) {\LARGE \( L_{i_0+1} \)};
\node at (0,-2) {\LARGE \( L_{i_0} \)};


\draw[thick] (5,-.5) -- (3,-.5) -- (0,0) -- (0,2) -- (3,2.5) -- (5,2.5);
\draw[thick] (7,2.5) -- (9,2.5) -- (12,1) -- (9,-.5) -- (7,-.5);

\draw[draw=none, fill=gray, fill opacity = 0.4, even odd rule, rounded corners=2mm ] (-1,3) rectangle ++(2,-4);
\draw[draw=none, fill=gray, fill opacity = 0.4, even odd rule, rounded corners=2mm] (2,3) rectangle ++(2,-4);
\draw[draw=none, fill=gray, fill opacity = 0.4, even odd rule, rounded corners=2mm] (8,3) rectangle ++(2,-4);
\draw[draw=none, fill=gray, fill opacity = 0.4, even odd rule, rounded corners=2mm] (11,3) rectangle ++(2,-4);

\end{tikzpicture}

\caption{The structure of holes in $G_0$ with respect to the levels $L_i$'s. 
}\label{fig:hole_forms}
\end{figure}
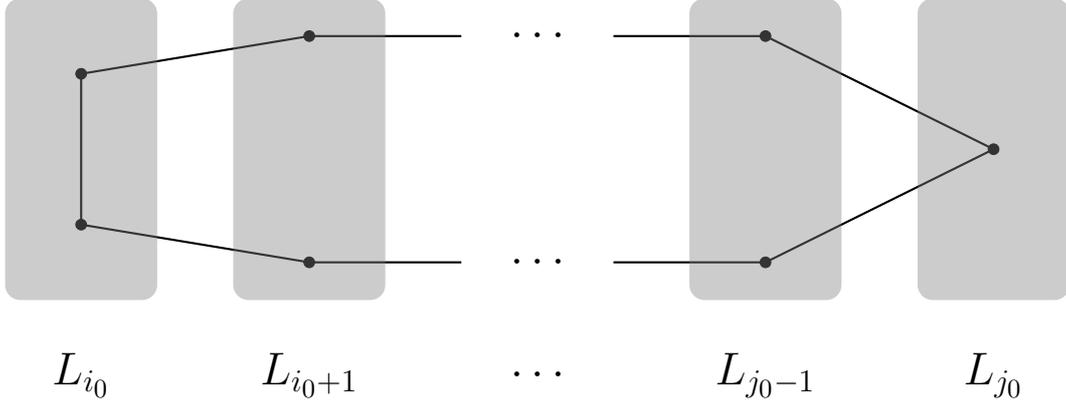

\begin{lem} \label{lem:structure_H}
	The hole \( H \) has exactly one vertex in its last level and exactly two vertices in its other levels.
\end{lem}
\begin{proof}
	First, we prove a claim.
	\begin{subclaim}\label{claim:subpath}
	Let $xyz$ be a subpath of the hole \(H\) such that $y\in L_i$, for some $i$, and $z\in L_{i+1}$. 
Then, if $i$ is not the first level of $H$, then $x\in L_{i-1}$ and otherwise $x\in L_i$.
	\end{subclaim}
To prove the claim, note that if $x\in L_{i+1}$, then by Lemma~\ref{lem:clique}, $x$ and $z$ are adjacent, which is a contradiction with the fact that $H$ is a hole. 
Also, if $i$ is not the first level, $x$ cannot be in $L_i$  due to Lemma~\ref{lem:edge_in_levels_H}-(i).
This proves the claim.
\\

Now, we prove the lemma. 
Let $x_0$ be a vertex of $H$ in its last level $L_{j_0}$. 
Then, by Lemma~\ref{lem:edge_in_levels_H}-(i), the two neighbors $x_1,x'_1$ of $x_0$ in $H$ are in $L_{j_0-1}$. 
Now, by Claim~\ref{claim:subpath} and Lemma \ref{lem:edge_in_levels_H}-(ii), the neighbors of $x_1$ and $x'_1$ in $V(H)\setminus\{x_0\}$ are in $L_{j_0-2}$. 
Finally, we have two distinct vertices $x_j,x'_j$ of $H$ in its first level $L_{i_0}$. 
We claim that $x_j$ and $x'_j$ are adjacent, which completes the proof. 
If $x_j$ and $x'_j$ are not adjacent, then $x_j$ and $x'_j$ have  neighbors $x_{j+1},x'_{j+1}$, respectively, in $V(H)\setminus \{x_0,x_1,\ldots, x_{j-1} \}\cup \{x'_1,\ldots,x'_{j-1}\}$ (that might be the same). 
By Claim~\ref{claim:subpath}, $x_{j+1}$ and $x'_{j+1}$ are in first level of $H$. 
So, $H$ contains at least two edges $x_jx_{j+1}$ and $x'_jx'_{j+1}$ in its first level which is a contradiction with Lemma~\ref{lem:edge_in_levels_H}-(ii). 
This completes the proof.
\end{proof}

Lemma~\ref{lem:structure_H} reveals the structure of the hole $H$ with respect to the levels $L_i$'s. 
So, $H$ has a configuration as Figure~\ref{fig:hole_forms}.

A hole $H$ is called \textit{dominated} if there is a vertex $v$ in $V\setminus V(H)$ such that $v$ is complete to $V(H)$. 
It is clear that in every claw-free graph, the only dominated hole is a cycle $C_5$ which is the induced subgraph of a wheel $W_5$. 
The next lemma states that every non-dominated hole $H$ has exactly one component in its first and last level and exactly two components in other levels. 
So, if \( H \) is not dominated, then \( |\tau(H)|= \ell-2 \), where $\ell$ is the length of $H$, and if $H$ is dominated, then $|\tau(H)|=2$.

\begin{lem} \label{lem:$2$-clique_C5}
	Let $H$ be a hole in $G_0$.
	If there is a component \( C \in \tau (H) \) in an inner level \( L_i \) such that \( V(H) \cap L_i \subset C \), then \( H \) is a \( C_5 \) which is dominated by a king vertex \( k \) of \( C \).
\end{lem}
\begin{proof}
	Let $H$ be a hole in $G_0$ with a component \( C \in \tau (H) \) in an inner level \( L_i \) such that \( V(H) \cap L_i \subset C \).
	Let $h_0,h_0',h_1,h_1',h_2,h_2'$ be vertices of $H$ such that \( h_0, h_0' \in V(H) \cap L_{i-1} \), \(  h_1, h_1' \in V(H) \cap L_{i} \) and \(  h_2, h_2' \in V(H) \cap L_{i+1} \)  such that \( h_0-h_1-h_2 \) and \( h_0'-h_1'-h_2' \) are subpaths of $H$. 
Note that if $i+1$ is the last level of $H$, then $h_2=h_2'$.
	Since \( C \) is a $2$-clique, by Lemma \ref{lem:backward_neighbours}-(i), there is a king vertex \( k \in V(C) \) such that \( k \) is adjacent to all vertices \( h_0, h_0',h_1, h_1'  \).
	Also, let \( h_2 - P - h_2' \) be the induced subpath of \( H \) in \( L_{>i} \).
	
	Now if \( h_2\neq h_2'\), then to prevent the hole \( h_1 - h_2 - P - h_2' - h_1' - k - h_1 \) to be an even hole, \( k \) must be adjacent to only one of \( h_2 \) or \( h_2' \).
	Without loss of generality, suppose that \( k \) is adjacent to \( h_2 \).
	Now \( \{k, h_0, h_1', h_2 \} \) forms a claw, a contradiction.
	So, \(  h_2 = h_2' \) and $L_{i+1}$ is the last level of $H$.
	Also, to prevent \( h_1 - h_2 - h_1' - k - h_1 \) to be an even hole, \( k \) is adjacent to \( h_2 \).
	
	Now, if \( h_0 \) and \( h_0' \) are non-adjacent, then \( \{k, h_0, h_0', h_2 \} \) forms a claw,  which is a contradiction. 
Hence, $h_0$ is adjacent to $h_0'$ and thus, \(  L_{i-1}  \) is the first level of \( H \) and \( H \) is a dominated \( C_5 \). 
This completes the proof.
\end{proof}

In the following lemma, we prove that every non-dominated hole $H$ in $G_0$ passes through king vertices of the components of \( H \).

\begin{lem} \label{lem:H_inner_level_king vertex}
Let $H$ be a hole of $G_0$. 
If $H$ is not a dominated  $C_5$, then every vertex of $H$ in a component \( C \in \tau (H) \) is a king vertex of \( C \). 
Also, if $H$ is a dominated $C_5$, then the vertex in the last level of $H$ is a king vertex.
\end{lem}
\begin{proof}
Let $H$ be an arbitrary hole in $G_0$. 
Suppose that $h_1\in V(H)\cap L_i$ is a non-king vertex of a component \( C \in \tau (H) \). 
Thus, \( C \) is a $2$-clique.
First, suppose that $L_i$ is the last level of $H$ (Here, $H$ could be either dominated or not).	
So, there are non-adjacent vertices \( h_0, h_2 \in V(H) \cap L_{i-1} \) which are both adjacent to $h_1$. 
	By Lemma \ref{lem:backward_neighbours}, there is a king vertex \( k \) of \( C \) such that \( k \) is adjacent to \( h_0,h_1,h_2 \) and a non-king vertex \( v \) of \( C \) such that \( v \) is adjacent to \( k \) but non-adjacent to \( h_0,h_1,h_2 \).
Now, \( \{k, h_0, h_2, v \} \) is a claw, a contradiction.

Now, suppose that $L_i$ is not the last level of $H$. 
Also, suppose that $H$ is not a dominated  $C_5$. 
Since $h_1$ is not a king vertex of \( C \), there is a king vertex $k$ of \( C \) and a non-king vertex $v$ of \( C \) such that $h_1$ is non-adjacent to $v$. 
Also, let $u$ be a common neighbor of $v$ and $k$ in $L_{i-1}$.
Let $h_0-h_1-h_2-P-h_2'-h_1'-h_0'$ be the subpath of $H$ in $L_{\geq{i-1}}$ such that $h_0,h_0'\in L_{i-1}$, $h_1,h_1'\in L_{i}$, and $h_2,h_2'\in L_{i+1}$ (if $L_{i+1}$ is the last level of $H$, then $h_2=h_2'$).

First, suppose that $k$ is non-adjacent to both  $h_2$ and $h_2'$. 
In this case, we claim that $Q=u-k-h_1-h_2-P-h_2'-h_1'-h_0'$ is a forbidden path. 
To see this, first note that since $H$ is not a dominated $C_5$, by Lemma~\ref{lem:$2$-clique_C5}, $h_1$ and $h_1'$ are in two distinct components of $H$ in $L_i$ and thus, $k$ is non-adjacent to $h_1'$. 
Therefore, by Lemma~\ref{lem:clique}, $k$ is non-adjacent to $h_0'$ and $u$ is non-adjacent to $h_1'$. 
This proves that $Q$ is a forbidden path, a contradiction with Lemma~\ref{lem:forbidden_path}.

Now, suppose that $k$ is adjacent to either $h_2$ or $h_2'$. 
Without loss of generality, suppose that $k$ is adjacent to $h_2$. 
Since $\{k,v,h_0,h_2\}$ does not induce a claw, we have $h_2$ is adjacent to $v$. 
Now, if $h_2=h_2'$, then $\{h_2,v,h_1,h_1'\}$ would induce a claw, and if $h_2\neq h_2'$, then $\{h_2,v,h_1,h_3\}$ would induce a claw, where $h_3 \in L_{i+2}$ is the first vertex in $P$ after $h_2$. 
This completes the proof.
\end{proof}

The following theorem, which is the main result of this section, illustrates the connection of two holes in \( G_0 \).

\begin{thm}\label{thm:tau}
	Let \( H_1 \) and \( H_2 \) be two holes in \( G_0 \).
	Then either \( \tau (H_1) \cap \tau (H_2) = \emptyset \), or \( \tau (H_1) = \tau (H_2) \).
	In the latter case, \( H_1 \) and \( H_2 \) share the same component in their first level.
\end{thm}

\begin{proof}
	Let \( C \in \tau (H_1) \cap \tau (H_2) \) be the common component of \( H_1, H_2 \) in the highest possible level \( L_i \). 
So, the components of $H_1$ and $H_2$ in $L_{>i}$ are disjoint.
Let \( h_{10} - h_{11} \) and \( h_{20} - h_{21}\) be the subpaths of \( H_1 \) and \( H_2 \), respectively, such that \( h_{11}, h_{21} \in C \) and \( h_{10}, h_{20} \in L_{i-1} \) (they exist since  \( L_i \) is not the first level of \( H_1 \) and \( H_2 \)). 
Also, for $j\in\{1,2\}$, let $h_{j2}$ be the unique neighbor of $h_{j1}$ in $V(H_j)\setminus \{h_{j0}\}$.
One of the following two cases may occur.	
	
	\paragraph{Case 1.} The vertices $h_{11}$ and  $h_{21}$ are distinct and non-adjacent.

	Then, $h_{11}$ and $h_{21}$ are not king vertices of $C$. 
Therefore, by Lemma~\ref{lem:H_inner_level_king vertex}, $H_1$ and $H_2$ are both dominated $C_5$ and $L_{i}$ is not the last level of both $H_1,H_2$. 
Let $h_{j1}'$, $j\in\{1,2\}$, be the unique vertex in $V(H_j)\cap L_i\setminus\{h_{j1}\}$.
Since $h_{11}$ and $h_{21}$ are non-adjacent vertices in the unique component  \( C \) of $H_1$ and $H_2$ in $L_i$, the set \( \{ h_{21}, h_{11}' \} \) is incomplete to the set \( \{ h_{11}, h_{21}' \} \) and $G_0$ induces on both sets a \( K_2 \).
Also, since the components of $H_1$ and $H_2$ in $L_{i+1}$ are disjoint,  $h_{12}$ is non-adjacent to $h_{22}$. 
So, by Lemma~\ref{lem:clique}, $h_{12}$ is incomplete to  $\{h_{21},h_{21}'\}$ and $h_{22}$ is incomplete to $\{h_{11},h_{11}'\}$. 
This implies that $h_{11}-h_{12}-h_{11}'-h_{21}-h_{22}-h_{21}'-h_{11}$ would be a hole $C_6$, a contradiction.

	\paragraph{Case 2.} The vertex \( h_{11}\) is adjacent to \( h_{21} \), or $h_{11}=h_{21}$.
	
	\paragraph{Subcase 2.1.} $h_{12},h_{22}\in L_{i+1}$.
	
	First, note that if $h_{11}=h_{21}$, then by Lemma \ref{lem:clique}, \( h_{12} \) is adjacent to \( h_{22} \), which is a contradiction with the choice of \( i \).
So, assume that $h_{11}$ and $h_{21}$ are distinct and adjacent.

	For $j\in \{1,2\}$, let \( h_{j0}', h_{j1}', h_{j2}' \) be the unique vertices in \( V(H_j) \cap L_{i-1} \setminus \{ h_{j0} \} \), \( V(H_j) \cap L_{i} \setminus \{ h_{j1} \} \), and \( V(H_j) \cap L_{i+1} \setminus \{ h_{j2} \}\), respectively.
	Also, let \( P^j \) be the subpath of \( H_j \) from \( h_{j2} \) to \( h_{j2}' \) in \( L_{> i} \).
	Note that, it is possible that \( h_{j2} = h_{j2}' \) whenever $L_{i+1}$ is the last level of $H_j$ (see Figure~\ref{fig:2.1}).
		We show that the path 
		\[
		Q = h_{11}' - h_{12}' - P^1 - h_{12} - h_{11} - h_{21} - h_{22} - P^2 - h_{22}' - h_{21}' 
		\]
		is a forbidden path.
		To prove this, we need the following claim.
		\begin{subclaim} \label{subclaim:case1.2}
			The subgraph \(P^1\cup P^2\) is an induced subgraph of $G$.
		\end{subclaim}
	Note that $P^1$ and $P^2$ are induced paths since they are subpaths of $H_1$ and $H_2$, respectively.
		To prove the claim, suppose that there are adjacent vertices \( x \in V(P^1) \) and \( y \in V(P^2) \).
		Without loss of generality, suppose that $x\in L_j$, for some $j$, and $y\in L_j\cup L_{j+1}$. 
If $y\in L_{j}$, then it contradicts the fact that \( L_i \) is the last level that \( H_1 \) and \( H_2 \) meet in a common component.
		So, $y\in L_{j+1}$. 
If $L_j$ is not the last level of $H_1$, then $x$ has a neighbor $x'\in L_{j+1}\cap V(H_1)$ and by Lemma~\ref{lem:clique}, $x'$ and $y$ are adjacent which is, again, contradicts how $L_i$ is chosen.
So, $L_j$ is the last level of $H_1$. 
Let $x',x''$ be the two neighbors of $x$ in $H_1$. 
Therefore, $x',x''\in L_{j-1}$ and thus, $\{x,x',x'',y\}$ induces a claw.
		This proves the claim.
		\\
		
		Note that every \( h \in \{ h_{11}, h_{11}', h_{21}, h_{21}' \} \) has only one neighbor in \( \{ h_{12}, h_{12}', h_{22}, h_{22}' \} \), since otherwise, by Lemma~\ref{lem:clique}, $P^1\cup P^2$ is not an induced subgraph, contradicting Claim~\ref{subclaim:case1.2}.		
		Also, \( h_{11}'\) is non-adjacent to  \( h_{21} \), since otherwise, \( h_{21} - h_{11} - h_{12}- P^1 - h_{12}' - h_{11}' - h_{21} \) is an even hole.
		Similarly, \( h_{11}\) is non-adjacent to  \( h_{21}' \). 
Hence, $Q$ is a forbidden path, contradicting Lemma~\ref{lem:forbidden_path}.

\begin{figure}
    \centering  
    \begin{subfigure}{0.4\textwidth}
        \centering
        \begin{tikzpicture}
\draw[fill=black] (3,6) circle (2pt);
\draw[fill=black] (0,6) circle (2pt);
\draw[fill=black] (3,4) circle (2pt);
\draw[fill=black] (0,4) circle (2pt);
\draw[fill=black] (0,2) circle (2pt);
\draw[fill=black] (3,2) circle (2pt);
\draw[fill=black] (3,0) circle (2pt);
\draw[fill=black] (0,0) circle (2pt);
\node at (0,-.5) {\Large \( h_{21}' \)};
\node at (0,1.5) {\Large \( h_{21} \)};
\node at (0,4.5) {\Large \( h_{11} \)};
\node at (0,6.5) {\Large \( h_{11}' \)};

\node at (3,-.5) {\Large \( h_{22}' \)};
\node at (3,1.5) {\Large \( h_{22} \)};
\node at (3,4.5) {\Large \( h_{12} \)};
\node at (3,6.5) {\Large \( h_{12}' \)};

\node at (0,-2) {\LARGE \( L_i \)};
\node at (3,-2) {\LARGE \( L_{i+1} \)};

\node at (5,1) {\LARGE \( P_1 \)};
\node at (5,5) {\LARGE \( P_2 \)};

\draw[thick] (0,0) -- (3,0);
\draw[thick] (3,0) .. controls (5, 0) and  (5, 2) .. (3,2);
\draw[thick] (3,2) -- (0,2) -- (0,4) -- (3,4);
\draw[thick] (3,4) .. controls (5, 4) and  (5, 6) .. (3,6);
\draw[thick] (3,6) -- (0,6);

\draw[draw=none, fill=gray, fill opacity = 0.4, even odd rule, rounded corners=2mm] (-1,7) rectangle ++(2,-8);
\draw[draw=none, fill=gray, fill opacity = 0.4, even odd rule, rounded corners=2mm] (2,7) rectangle ++(2,-8);

\end{tikzpicture}  
        \caption{Subcase 2.1}\label{fig:2.1}
    \end{subfigure}\hfill
    \begin{subfigure}{0.6\textwidth}
        \centering
        \begin{tikzpicture}

\draw[fill=black] (0,4) circle (2pt);
\draw[fill=black] (0,2) circle (2pt);
\draw[fill=black] (3,2) circle (2pt);
\draw[fill=black] (3,0) circle (2pt);
\draw[fill=black] (0,0) circle (2pt);

\draw[fill=black] (-3,6) circle (2pt);
\draw[fill=black] (-3,4) circle (2pt);

\draw[fill=black] (-3,0) circle (2pt);


\node at (-3,-.5) {\Large \( h_{20}' \)};

\node at (-3,4.5) {\Large \( h_{10} \)};
\node at (-3,6.5) {\Large \( h_{12} \)};

\node at (0,-.5) {\Large \( h_{21}' \)};
\node at (0,1.5) {\Large \( h_{21} \)};
\node at (0,4.5) {\Large \( h_{11} \)};

\node at (3,-.5) {\Large \( h_{22}' \)};
\node at (3,1.5) {\Large \( h_{22} \)};

\node at (-3,-2) {\LARGE \( L_{i-1} \)};
\node at (0,-2) {\LARGE \( L_i \)};
\node at (3,-2) {\LARGE \( L_{i+1} \)};

\node at (5,1) {\LARGE \( P_1 \)};

\draw[thick] (-3,0) --(0,0) -- (3,0);
\draw[thick] (3,0) .. controls (5, 0) and  (5, 2) .. (3,2);
\draw[thick] (3,2) -- (0,2) -- (0,4) -- (-3,4);
\draw (-3,6) -- (0,4);

\draw[dashed] (-3,4) -- (0,2);

\draw[draw=none, fill=gray, fill opacity = 0.4, even odd rule, rounded corners=2mm] (-4,7) rectangle ++(2,-8);
\draw[draw=none, fill=gray, fill opacity = 0.4, even odd rule, rounded corners=2mm] (-1,7) rectangle ++(2,-8);
\draw[draw=none, fill=gray, fill opacity = 0.4, even odd rule, rounded corners=2mm] (2,7) rectangle ++(2,-8);

\end{tikzpicture}
 
        \caption{Subcase 2.2}        	\label{fig:2.2}
    \end{subfigure}
\caption{An schema of Subcases 2.1 and 2.2 in the proof of Lemma~\ref{thm:tau}.}
\end{figure}
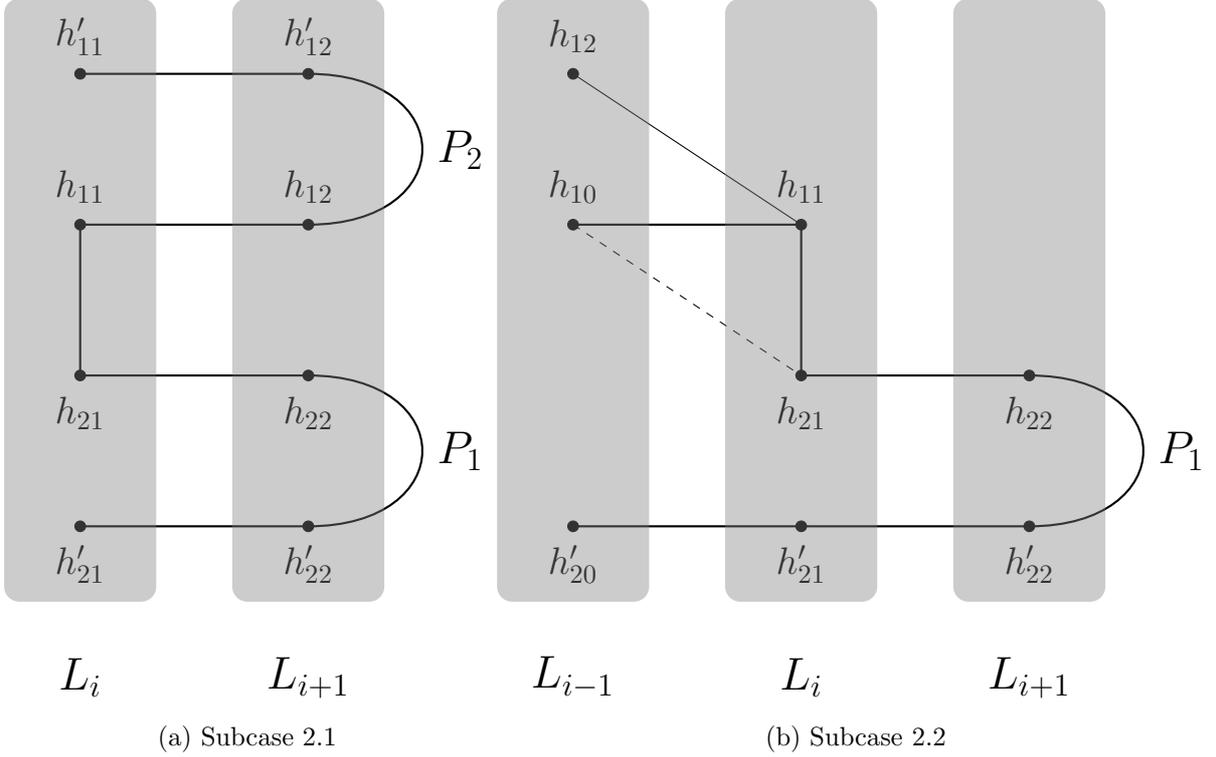

	\paragraph{Subcase 2.2.}
	$h_{12}\in L_{i-1},h_{22}\in L_{i+1}$.
	
First, note that if $h_{11}=h_{21}$, then we have a claw on the set of vertices $\{h_{11},h_{22},h_{10},h_{12}\}$. 
So, assume that $h_{11}$ and $h_{21}$ are distinct and adjacent.
Let \( h_{20}', h_{21}', h_{22}' \) and the path $P^2$ be as defined in Subcase~2.1 (see Figure~\ref{fig:2.2}).
In this case, $h_{21}$ is non-adjacent to either $h_{10}$ or $h_{12}$, since otherwise $\{h_{21},h_{10},h_{12},h_{22} \}$ would be a claw. 
Without loss of generality, suppose that $h_{21}$ is non-adjacent to $h_{10}$. 
We claim that $Q= h_{10} - h_{11} - h_{21} - h_{22} - P^2 - h_{22}' - h_{21}' - h_{20}' $ is a forbidden path.
To see this, first note that $h_{11}$ has no neighbor $v$ in $L_{>i}$, since otherwise, $\{h_{11},v,h_{10},h_{12}\}$ would be a claw.
Also, $h_{11}$ is non-adjacent to $h'_{21}$, since otherwise, $h_{11}- h_{21} - h_{22} - P^2 - h_{22}' - h_{21}' - h_{11}$ would be an even hole. 
So, by Lemma~\ref{lem:clique}, $h_{10}$ is non-adjacent to $h_{21}'$ and $h_{20}'$ is non-adjacent to $h_{11}$. 
Hence, $Q$ is a forbidden path, contradicting Lemma~\ref{lem:forbidden_path}.

\paragraph{Subcase 2.3.} 	$h_{12},h_{22}\in L_{i-1}$.

In this subcase, $i$ is the last level of $H_1$ and $H_2$. Suppose that $\tau(H_1)\neq \tau(H_2)$ and let $L_j$ be the level with the largest index such that there exists a component $C\in \tau(H_1)$ in the level $L_j$ where $V(C)\cap V(H_2)=\emptyset$. 
If $j=i-1$, then either $h_{10}$ or $h_{12}$, say $h_{10}$, is in $C$.
Thus, $h_{10}$ is non-adjacent to $h_{20}$ and $h_{22}$ (Figure~\ref{fig:2.3}). 
\begin{figure}
	\centering
	\begin{tikzpicture}
\draw[fill=black] (0,6) circle (2pt);
\draw[fill=black] (3,4) circle (2pt);
\draw[fill=black] (0,4) circle (2pt);
\draw[fill=black] (0,2) circle (2pt);
\draw[fill=black] (3,2) circle (2pt);
\draw[fill=black] (0,0) circle (2pt);

\draw[fill=black] (-4,2) circle (2pt);
\draw[fill=black] (-7,2) circle (2pt);

\draw[fill=black] (-4,4) circle (2pt);
\draw[fill=black] (-7,4) circle (2pt);
\draw[fill=black] (-10,4) circle (2pt);

\node at (0,-.5) {\Large \( h_{22} \)};
\node at (0,1.5) {\Large \( h_{20} \)};
\node at (0,4.5) {\Large \( h_{12} \)};
\node at (0,6.5) {\Large \( h_{10} \)};

\node at (-4,1.5) {\Large \( v_{j+2} \)};
\node at (-7,1.5) {\Large \( v_{j+1} \)};

\node at (-4,4.5) {\Large \( u_{j+2} \)};
\node at (-7,4.5) {\Large \( u_{j+1} \)};
\node at (-10,4.5) {\Large \( u_{j} \)};

\node at (3,1.5) {\Large \( h_{21} \)};
\node at (3,4.5) {\Large \( h_{11} \)};

\node at (3,-2) {\LARGE \( L_i \)};
\node at (0,-2) {\LARGE \( L_{i-1} \)};
\node at (-2,-2) {\LARGE \( \cdots \)};
\node at (-4,-2) {\LARGE \( L_{j+2} \)};
\node at (-7,-2) {\LARGE \( L_{j+1} \)};
\node at (-10,-2) {\LARGE \( L_{j} \)};


\draw[thick] (0,0) -- (3,2);
\draw[thick] (3,2) -- (0,2);
\draw[thick] (0,4) -- (3,4);
\draw[thick] (3,2) -- (3,4);

\draw[thick] (3,4) -- (0,6);
\draw[thick] (-10,4) -- (-7,4) -- (-4,4);
\draw[thick] (-7,2) -- (-4,2);

\draw[draw=none, fill=gray, fill opacity = 0.4, even odd rule, rounded corners=2mm] (-11,7) rectangle ++(2,-8);
\draw[draw=none, fill=gray, fill opacity = 0.4, even odd rule, rounded corners=2mm] (-8,7) rectangle ++(2,-8);
\draw[draw=none, fill=gray, fill opacity = 0.4, even odd rule, rounded corners=2mm] (-5,7) rectangle ++(2,-8);
\draw[draw=none, fill=gray, fill opacity = 0.4, even odd rule, rounded corners=2mm] (-1,7) rectangle ++(2,-8);
\draw[draw=none, fill=gray, fill opacity = 0.4, even odd rule, rounded corners=2mm] (2,7) rectangle ++(2,-8);

\end{tikzpicture}
	\caption{An schema of Subcase 2.3 in the proof of Lemma~\ref{thm:tau}.}
	\label{fig:2.3}
\end{figure}
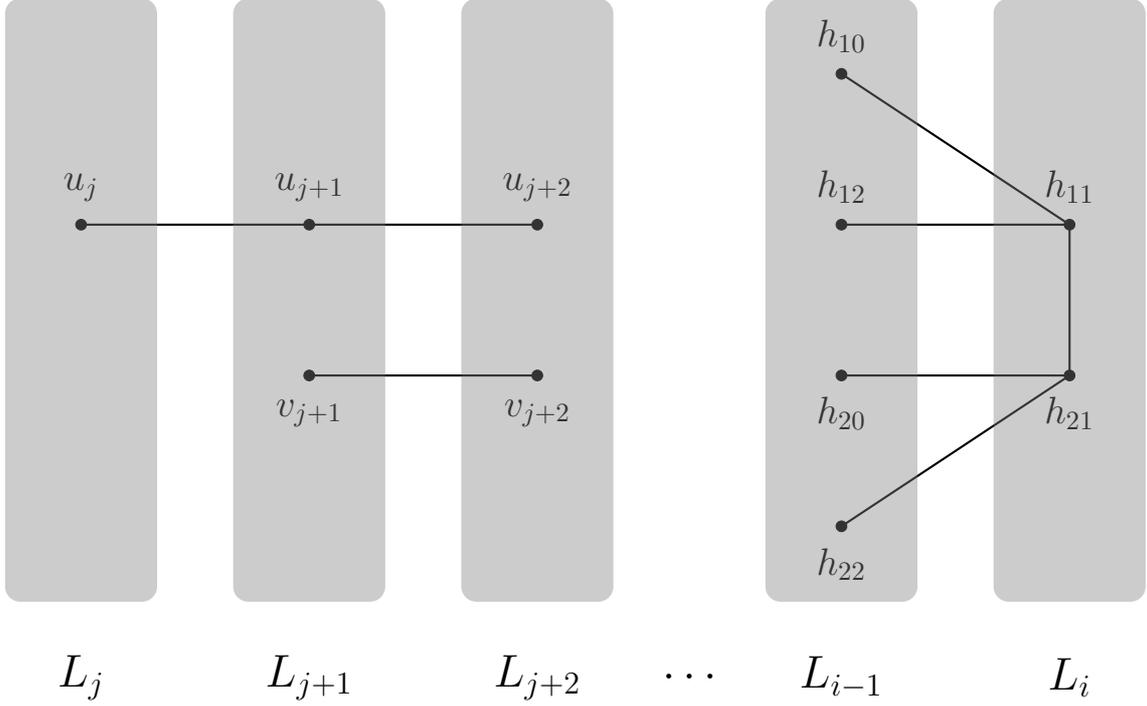
So, in case $h_{11}=h_{21}$, we have a claw on the set $\{h_{21},h_{20},h_{22},h_{10}\}$, and, in case $h_{11}$ is adjacent to $h_{21}$, by Corollary~\ref{col:xy_edge}, either $h_{11}$ is adjacent to both $h_{20}$ and $h_{22}$, or $h_{21}$ is adjacent to $h_{10}$ which yields that there is a claw with leaves $h_{10},h_{20},h_{22}$ and the center $h_{11}$ or $h_{21}$.

So, suppose that $j\leq i-2$ and thus $H_1\neq C_5$. 
Take a vertex $u_j\in V(H_1)\cap V(C)$ and let $u_j-u_{j+1}-u_{j+2}$ be the subpath of $H_1$ in $L_{j}\cup L_{j+1}\cup L_{j+2}$. By the choice of $j$, there is a vertex $v_{j+1}\in V(H_2)\cap L_{j+1}$ which is in the same component with $u_{j+1}$. Let $v_j-v_{j+1}-v_{j+2}$ be the subpath of $H_2$ where $v_{j+2}\in L_{j+2}$. Then we have  $v_j\in L_j$, since otherwise $L_{j+1}$ is the first component of $H_2$ and $H_1$ has two components in $L_{j+1}$, this is in contradiction with the choice of $j$. Thus, $v_j\in L_j$ and also $v_j\not\in C$ (by the definition of $C$).
Now, if $u_{j+1}=v_{j+1}$, then we have a claw on $\{u_{j+1},u_{j},v_{j}, u_{j+2}\}$. 
So, $u_{j+1}$ and $v_{j+1}$ are distinct. 
Also, since $H_1$ is not $C_5$, by Lemma~\ref{lem:H_inner_level_king vertex}, $u_{j+1}$ is adjacent to $v_{j+1}$. 
Now, by Corollary~\ref{col:xy_edge}, either $u_{j+1}$ is adjacent to $v_{j}$, or $v_{j+1}$ is adjacent to $u_{j}$ and we would have a claw on $\{u_{j+1},u_{j},v_{j},u_{j+2}\}$ or $\{v_{j+1},u_{j},u_{j},v_{j+2}\}$.

\vspace{10pt}

This proves that $\tau(H_1)=\tau(H_2)$. 
Now, we prove that, in this case, the first components of $H_1$ and $H_2$ are the same. 
Let $L_{i_0}$ be the first level of $H_1$ and $H_2$. 
Also, suppose that \( h_{11}, h_{11}' \in V(H_1) \cap L_{i_0} \) and \( h_{21}, h_{21}' \in V(H_2) \cap L_{i_0} \) are the vertices of \( H_1 \) and \( H_2 \) in their first level, respectively, which, by contrary, are located in different level components.
For $j\in \{1,2\}$, let \( h_{j1}-h_{j2}-h_{j3}\) and \(h_{j1}'-h_{j2}'-h_{j3}'\) be the subpaths of $H_j$ in $L_{i_0}\cup L_{i_0+1}\cup L_{i_0+2} $. 
(Note that we might have $h_{j3}=h_{j3}'$ if $H_j$ is a cycle $C_5$).
Since $\tau(H_1)=\tau(H_2)$, without loss of generality, suppose that \( h_{12}, h_{22} \) and \( h_{12}', h_{22}' \) are located in the  level components \( C_1 \) and \( C_2 \) in $L_{i_0+1}$, respectively.
Note that \( h_{12}\) is non-adjacent to \( h_{21} \), since otherwise, \( \{ h_{12}, h_{11}, h_{21}, h_{13} \} \) woud be a claw.
Similarly, \( h_{12}'\) is non-adjacent to \( h_{21}' \). 
(This implies that $h_{12}\neq h_{22}$ and $h_{12}'\neq h_{22}'$.)

If \(H_1\) and \(H_2\) are not dominated $C_5$, then by Lemma~\ref{lem:H_inner_level_king vertex}, \( h_{12} \) is adjacent to \( h_{22} \) and \( h_{12}' \) is adjacent to \( h_{22}' \). 
Also, if  \(H_1\) and \(H_2\) are dominated $C_5$, then \( C_1 = C_2 \) and 
so again, \( h_{12} \) is adjacent to \( h_{22} \) and \( h_{12}' \) is adjacent to \( h_{22}' \) and the set \( \{ h_{12}, h_{22} \} \) is incomplete to  the set \( \{ h_{12}', h_{22}' \} \).
Therefore, in both cases, \( h_{11} - h_{12} - h_{22} - h_{21} - h_{21}' - h_{22}' - h_{12}' - h_{11}' \) is an induced cycle \( C_8 \), a contradiction.
\end{proof}

\section{Winning Strategy with Three Cops}\label{sec:algorithm}

In this section, we prove that the cop number of every (claw, even-hole)-free graph is at most $3$.
In the next section, we refine the algorithm to find a winning strategy that needs at most two cops.
First, we have to employ the results of the previous sections to prove a couple of lemmas.
We start by showing how paths connect the level components.

\begin{lem}\label{lem:mate}
Let $C$ be a level component in $L_i$.

\begin{itemize}
\item[\rm (i)] We have two possibilities:
\begin{itemize}
	\item[\rm (a)] There is a unique level component $C'\neq C$ in $L_i$ for which there is a path from $C$ to $C'$ whose vertices are in $L_{\geq i}$. 
In this case, we say that $C$ is non-singular.
	\item[\rm (b)] There is no path from $C$ to any other level component of $L_i$ whose vertices are in $L_{\geq i}$. In this case, we say that $C$ is singular.
\end{itemize}
When $C$ is non-singular, we say that $C'$ is the \textit{mate} of $C$ and we denote $C'$ by $\eta(C)$. 
Also, when $C$ is singular, we define $\eta(C)=C$. 

\item[\rm (ii)] If $C$ is connected to a  level component $\widehat{C}$ in $L_{i+1}$, then the set of neighbors of $\widehat{C}$ in $L_i$ is a subset of $V(C)\cup V(\eta(C))$ and in case $C$ and $\widehat{C}$ are both non-singular, $\widehat{C}$  has no neighbor in $\eta(C)$.
\end{itemize}
\end{lem}

For instance, for every hole $H$ in $G_0$, the component of $H$ in its last level is singular. 
Also, if $H$ is not a dominated $C_5$, all of its components in its inner levels are non-singular and are singular otherwise.

\begin{proof}
In order to prove (i), for the contrary, suppose that there are two distinct level components $C_1,C_2$ in $L_i$ such that two induced paths $P^1$ and $P^2$ exist from $C$ to $C_1$ and $C_2$, respectively, whose all vertices are in $L_{\geq i}$. 
Using the method given in the proof of Lemma~\ref{lem:forbidden_path}, one can extend $P^1$ and $P^2$ into the holes $H_1$ and $H_2$ such that $C\in \tau(H_1)\cap \tau(H_2)$, but $C_1\in \tau(H_1)\setminus \tau(H_2)$, a contradiction with Theorem~\ref{thm:tau}. 
This proves (i).

In order to prove (ii), suppose that $\widehat{C}$ is connected to a level component $C'$ in $L_{i}$. 
Then, there is a path from $C$ to $C'$ whose vertices are in $L_{\geq i}$. 
Therefore, either $C'=C$, or $C'=\eta(C)$. 
This implies that the neighbors of \( \widehat{C} \) in $L_i$ are in $ V(C)\cup V(\eta(C))$. 
Now, suppose that both $C$ and $\widehat{C}$ are non-singular and $\widehat{C}$ is connected to both $C$ and $\eta(C)$. 
So, there is an induced path $\widehat{P}$ from $\widehat{C}$ to $\eta(\widehat{C})$ with vertices in $L_{\geq i+1}$ and there is an induced path $P$ which starts in $C$, passes $\widehat{C}$ and ends in $\eta(C)$. 
Again, $P$ and $\widehat{P}$ can be extended to the holes $H$ and $\widehat{H}$ such that $\widehat{C}\in  \tau(H)\cap \tau(\widehat{H})$, but $\eta(\widehat{C})\in \tau(\widehat{H})\setminus \tau(H)$, a contradiction with Theorem~\ref{thm:tau}. 
This proves (ii).

\end{proof}

 \subsection{The Algorithm} 

By a \textit{simple path}, we mean a path that has at most one vertex in each level $L_i$.
Let \( G \) be a connected (claw, even-hole)-free graph.
We describe our algorithm for the game of cops and robber in the graph $G$. 
In the beginning, we put three cops $c_0,c_1,c_2$ on an arbitrary vertex $u_0$. 
If the robber is located at a vertex $v$, then let $P^0$ be a shortest path from $u_0$ to $v$ and let $u_1$ be the neighbor of $u_0$ in $P^0$. 
Also, define $B_0$ and $G_0$ as in Subsection~\ref{sub:notation}. 
We always keep the cop $c_0$ on $u_0$, so the robber will be captured once it enters the box $B_0$. 
So, we may assume that the game is played in the graph $G_0$. 
Now, we move two cops $c_1$ and $c_2$ to $u_1$ and during the algorithm, we guarantee that for every simple path $P$ from the robber's location to $u_1$ in \( G_0 \), $P$ passes through a level component containing $c_1$ or $c_2$ (see Lemma~\ref{lem:correctness}).
For $i\in\{1,2\}$, let $C_i$ be the level component containing $c_i$, where $C_1$ and $C_2$ are both in the level $L_j$, for some $j\geq 1$. 
If there is a simple path from $u_1$ to the robber's location passing through a vertex $v\in C_i$, for some $i\in\{1,2\}$, we take one of these paths arbitrarily, say $P$, and define $C_i^+$ as the level component in $L_{j+1}$ containing the vertex $v^+$ which is the next vertex after $v$ in $P$.
Now, we define the cops' next steps as follows.

\begin{itemize}
	\item \textbf{Progressing:} If $C_{i+1}$ is connected to $\eta(C_i^+)$, then we move $c_i$ from $C_i$ to a king vertex of $C_i^+$ and $c_{i+1}$ from $C_{i+1}$ to a king vertex of $\eta(C_i^+)$ (subscripts are read modulo $2$). 
This can be safely done in at most two rounds (see Lemma~\ref{lem:safe_moves}). 
Note that if both $C_1^+$ and $C_2^+$ are defined, then by the definition, $C_2^+=\eta(C_1^+)$.
	\item \textbf{Synchronizing:} If $C_{i+1}$ is incomplete to $\eta(C_i^+)$, then we keep $c_i$ in its position and during a number of rounds, we move $c_{i+1}$ from $C_{i+1}$ to a king vertex of $C_{i}$ using a path with vertices in $L_{\geq j}$ (subscripts are read modulo $2$).
\end{itemize}

Note that after a progressing step, the level components containing $c_1,c_2$ are mate and after a synchronizing step, $c_1,c_2$ are in the same level component. 
Finally, the algorithm for the cops to chase the robber can be summarized as follows.

\begin{algorithm}
\begin{enumerate}
	\item \textbf{Initialization.} Locate the cops $c_0,c_1,c_2$ on an arbitrary vertex $u_0$ and let $P^0=u_0-u_1-\cdots$ be a path from $u_0$ to the robber's position. 
Then, move $c_1$ and $c_2$ to $u_1$.

	\item In \textit{each} round, if the robber is adjacent to some cop, then he is captured and the game is over. 
Otherwise,
	\item Suppose that the cops $c_1,c_2$ are located on the vertices $k_1,k_2$ in the level components $C_1,C_2$, respectively. 
Let $P$ be a simple path from the robber's location to $u_1$ which passes through $C_i$, for some $i\in\{1,2\}$. 
Then,
	\begin{itemize}
		\item  If $C_{i+1}$ is connected to $\eta(C_i^+)$, then do the progressing step as follows:
		Let $k_1',k_2'$ be king vertices of $C_i^+$ and $\eta(C_i^+)$, respectively. 
Then, for each $j\in \{1,2\}$, if $k_j$ is adjacent to $k'_j$, then move $c_j$ from $k_j$ to $k_j'$. 
Otherwise, find a vertex $v_j$ in $C_j$ which is adjacent to $k_j'$ and move $c_j$ from $k_j$ to $v_j$ and then, in the next round, from $v_j$ to $k_j'$.

		\item  If $C_{i+1}$ is incomplete to $\eta(C_i^+)$, then do the synchronizing step.
	\end{itemize}
	\end{enumerate}
	\caption{The algorithm}\label{alg:main}
\end{algorithm}

\subsection{Correctness}\label{sub:correct}
Here, we are going to prove that by applying Algorithm~\ref{alg:main}, the cops capture the robber in a finite number of rounds. 
First, we need a couple of lemmas.
%
\begin{lem} \label{lem:correctness}
During the execution of Algorithm~\ref{alg:main}, every simple path from the robber's location to $u_1$ passes through a level component containing $c_1$ or $c_2$.
\end{lem}
\begin{proof}
Suppose that $c_i$ is in the level component $C_i$, $i\in \{1,2\}$, both in the level $L_j$. 
We use induction on $j$. 
If $j=1$, then $c_1$ and $c_2$ are both on $u_1$ and we are done. 
Now,  
 we are going to prove that after (and during) a progressing or synchronizing step, every simple path from the robbers' location to $u_1$ passes through a level component containing $c_1$ or $c_2$. 
Without loss of generality, assume that $C_1^+$ is defined. 
So, there is a path $P$ from the robbers' location to $u_1$ which passes through $C_1$ in $L_j$ and $C_1^+$ in $L_{j+1}$.

First, assume that we are going to perform a  progressing step. 
So, $C_2$ is connected to  $\eta(C_1^+)$ and we move $c_1$ from $C_1$ to $C_1^+$ and $c_2$ from $C_2$ to $\eta(C_1^+)$. 
On the other hand, if $P'$ is a simple path from the robbers' new location to $u_1$, then $P'$ passes through $C_1^+$ or $\eta(C_1^+)$ in $L_{j+1}$. 
To see this, note that if $P'$ passes through a level component $C'$ in $L_{j+1}$, then there is a path from $C_1^+$ to $C'$ and so by, Lemma~\ref{lem:mate}-(i), either $C'=C_1^+$, or $C'=\eta(C_1^+)$. 
This shows that $P'$ passes through a level component  containing $c_1$ or $c_2$ in $L_{j+1}$.

Now, suppose that we are going to perform a  synchronizing step. 
So, $C_2=\eta(C_1)$  is incomplete to  $\eta(C_1^+)$; we keep $c_1$ in $C_1$ and move $c_2$ to $C_1$. 
During these rounds, suppose that $P'$ is a simple path from the robbers' new location to $u_1$ which passes through level components $C$ in $L_j$ and $C'$ in $L_{j+1}$. 
We claim that $C_1=C$. 
To see this, note that there is a path from $C_1^+$ to $C'$ in $L_{\geq j+1}$, so $C'\in \{C_1^+, \eta(C_1^+)\}$ and $C\in \{C_1, C_2\}$. 
For the contrary, suppose that $C=C_2\neq C_1$. 
The case $C'=\eta(C_1^+)$ is in contradiction with the fact that $C_2$ is incomplete to $\eta(C_1^+)$ and the case $C'=C_1^+\neq \eta(C_1^+)$ is in contradiction with Lemma~\ref{lem:mate}-(ii) since $C_1^+$ is connected to both $C_2$ and $\eta(C_2)=C_1$. 
This proves that $C=C_1$ and $P'$ passes through the level component containing $c_1$.

\end{proof}

The next lemma guarantees that the progressing step can be done safely; the robber cannot run off to the lower levels during this step.

\begin{lem} \label{lem:safe_moves}
	Suppose that we are in a progressing step and for each $i\in\{1,2\}$, the cop $c_i$ is in the king vertex of the level component $C_i$ in $L_j$. 
Then, \( c_1,c_2 \) can move safely to a king vertex of \( C_1^+ \) and $\eta(C_1^+)$ in at most two rounds, in the sense that during the movement, the robber cannot move to the level $L_j$ unless he is captured.
\end{lem}
\begin{proof}
		Let \( c_1,c_2 \) be the cops on the king vertices \( k_1,k_2 \) in the level components \( C_1,C_2 \) in \( L_{j} \) and the goal is to move \( c_1,c_2 \) to  king vertices \( k_1',k_2' \) of the  level component \( C_1^+ \) and $\eta(C_1^+)$ in \( L_{j+1} \), respectively.
				
		Since there is an edge between $C_1$ (resp. $C_2$) and $C_1^+$ (resp. $\eta(C_1^+)$), by Lemma~\ref{lem:backward_neighbours}, $k_1'$ (resp. 
$k_2'$) is adjacent to a vertex $v_1$ (resp. $v_2$) in $C_1$ (resp. $C_2$). 
So, for each $i\in\{1,2\}$, if $k_i$ is adjacent to $k_i'$, then simply move $c_i$ from $k_i$ to $k_i'$, and otherwise, move $c_i$ from $k_i$ to $v_i$ and then, in the next round, from $v_i$ to $k_i'$.
		Now, we show that, during the process, the robber cannot move to the level $L_j$ unless he is captured.

		If the robber reaches the level $L_j$, then by Lemma~\ref{lem:correctness}, the robber moves to either $C_1$ or $C_2$. 
Also, by Lemma~\ref{lem:mate}-(ii), the robber moves either from $C_1^+$ to $C_1$, or $\eta(C_1^+)$ to $C_2$. 
Without loss of generality, we assume that the robber moves from the vertex $r'$ in $C_1^+$ to the vertex $r$ in $C_1$.		
%
		First, suppose that $k_1$ is non-adjacent to $k_1'$.
		At the time of arrival of the robber to \( r \), if \( c_1 \) is located on vertex \( v_1 \)  then \( v _1\) should be non-adjacent to \( r \) since otherwise, the robber would be captured.
		Additionally, \( k_1' \) is non-adjacent to \( r \) since otherwise \( v_1 - k_1' - r - k_1 - v_1 \) would be an induced \( C_4 \).
		By the same reason, \( r' \) is non-adjacent to \( v_1 \).
		So, \( k_1' \) and \( r' \) are distinct.
		But now, \( v_1 - k_1' - r' - r \) is a forbidden path, a contradiction.
		If \( c_1 \) is located in the vertex \( k_1' \) at the time of arrival of the robber to \( r\), then to prevent  \( c_1 \) to capture the robber in the next round, \( k_1'  \) is non-adjacent  to \( r \) and so \( r' \) and \( k_1' \) are distinct.
		Also, \( v_1 \) is non-adjacent to \( r' \) since otherwise the robber must had been captured by \( c_1 \) in its last move.
But, again, \( v_1 - k_1' - r' - r \) is a forbidden path.
		Finally, if $k_1$ is adjacent to $k_1'$, then $c_1$ moves directly from $k_1$ to $k_1'$. 
If, after this move, the robber moves from $r'$ to $r$, then again $k_1-k_1'-r'-r$ would be a forbidden path. 
This completes the proof.	
\end{proof}

Now, we are ready to prove the correctness of the algorithm.

\begin{thm}
By executing Algorithm~\ref{alg:main}, the cops can capture the robber in at most $n$ rounds.
\end{thm}

\begin{proof}
	 Suppose that, at a moment, the cops are in the level $L_j$, for some $j\geq 1$. By Lemma~\ref{lem:correctness}, the robber can reach to the level $L_j$ only by entering to the level components containing $c_1$ or $c_2$. 
So, if the cops are on the king vertices of their level components, then they can capture him. 
Otherwise, the cops are performing a progressing step, and again they can capture him by Lemma~\ref{lem:safe_moves}.
Hence, it is remained to prove that, at each moment, we always have a progressing step after a finite number of rounds. 
To see this, note that the first step is a progressing step which moves $c_1,c_2$ from $u_1$ to the level $L_2$. 
Suppose that after a progressing step, we have a synchronizing step. 
We claim that the next step is surely a progressing step. 
Assume that after the progressing step, $c_1,c_2$ are in the level component $C_1$ and $C_2=\eta(C_1)$, respectively and $C_2$ is incomplete to $\eta(C_1^+)$. 
So, we perform a synchronizing step and move $c_2$ to $C_1$. 
Now, $c_1,c_2$ are both in $C_1$. 
We claim that $C_1$ is connected to $\eta(C_1^+)$, and so the next step is a progressing step. 
To see this, note that the level component $\eta(C_1^+)$ is connected to a level component $C$ in $L_j$. 
Since there is a path from $C_1$ to $C$ in $L_{\geq j}$, we have either $C=C_1$ or $C=\eta(C_1)=C_2$. 
Now, since $C_2$ is incomplete to $\eta(C_1^+)$, we have $C=C_1$ and so $C_1$ is connected to $\eta(C_1^+)$ and the next step is a progressing step. 
This shows that the algorithm ends in a finite number of rounds, and the cops eventually capture the robber.

Now, we prove that the capture time is at most $n$. 
To do this, we prove that at each round, $c_1$ or $c_2$ visits a new vertex which has not been visited before. 
First, we claim that the vertices visited through a synchronizing step will never be seen afterward. 
To see this, suppose that $c_1$ and $c_2$ are respectively in the level components $C_1$ and $C_2$ in the level $L_j$ and we start a synchronizing step by moving $c_2$ from $C_2$ to $C_1$ through the path $P$ in $L_{\geq j}$. 
Also, suppose that $c_1$ or $c_2$ will visit an inner vertex of $P$ say $v$ in a forthcoming round. 
Also, let $C$ be the level component in $L_{j+1}$ which intersects \( P \) and is connected to $C_2$. 
So, there is a path from $C_1^+$ or $\eta(C_1^+)$ to $C$ in $L_{\geq j+1}$. 
On the other hand, since we are in a synchronizing step, $C\neq \eta(C_1^+)$. 
Also, by Lemma~\ref{lem:mate}-(ii), $C\neq C_1^+$. 
This is a contradiction with Lemma~\ref{lem:mate}-(i). 
This proves the claim. 

Now, note that in every progressing step, we move the cops to vertices in the next level, which, by the claim, are not visited before in previous synchronizing steps. 
Also, in a synchronizing step, the cop moves on a path with unvisited vertices. 
This proves that, in each round, at least one cop moves to an unvisited vertex and thus, the capture time is at most $n$. 
Note that, in the last round of a synchronizing step, when $c_2$ enters the level component containing $c_1$, we might see an already visited vertex. However, this loss has already been compensated since the previous step was a progressing step in which we have seen two unvisited vertices in a single round.
\end{proof}

\section{The Refinement} \label{sec:refinement}
In this section, we are going to improve Algorithm~\ref{alg:main} to prove that two cops are enough to capture the robber in (claw, even-hole)-free graphs.
In fact, in Algorithm~\ref{alg:main}, we always keep a cop on $u_0$ to deter the robber from accessing to $V(G)\setminus V(G_0)$. 
Here, we prove that it is unnecessary to keep a cop on \( u_0 \), and thus two cops are sufficient.
To do this, first, we need a couple of lemmas. 
The following lemma shows how the connections are between the vertices in $B_0$ and $V(G_0)$ (for the definition of $B_0$ and $G_0$, see Subsection~\ref{sub:notation}).

\begin{lem}\label{lem:box}
	For every vertex $b\in B_0$,
	\begin{itemize}
		\item[\rm (i)]
		the set of neighbors of $b$ in $V(G_0)\setminus \{u_0,u_1\}$ is a subset of two consecutive levels of $G_0$ and, forms a clique.

			\item [\rm (ii)]
		Let $b$ be adjacent to a vertex $v\in V(G_0)\setminus \{u_0,u_1\}$. 
Then $b$ is either complete or incomplete to $N^-(v)$ and, is complete to either $N^+(v)$ or $N^-(v)$.
	\end{itemize}
	
\end{lem}
\begin{proof}
To prove (i), if $b$ has two non-adjacent neighbors $v_1,v_2$ in $V(G_0)\setminus \{u_0,u_1\}$, then $G$ induces a claw on $\{b,u_0,v_1,v_2\}$ which is a contradiction. 
So, the set of neighbors of $b$ in $V(G_0)\setminus \{u_0,u_1\}$ is a clique and consequently is a subset of two consecutive levels.

To prove (ii), suppose that $b$ is adjacent to a vertex $v_1$ in $N^-(v)$ and is non-adjacent to a vertex $v_2$ in $N^-(v)$. 
Also, let  $P^1$ and $P^2$ be two simple paths from $v_1$ and $v_2$, respectively, to $u_1$. 
Also, Let $P^0$ be an induced path of length at most two from $u_1$ to $b$ (i.e. either $P^0=u_1-b$ or $P^0=u_1-u_0-b$).

Then, $G$ has two induced cycles $b-v_1-P^1-P^0-b$ and $b-v-v_2-P^2-P^0-b$ which are of different parities, a contradiction (note that by (i), \( b \) is non-adjacent to inner vertices of \( P^1 \) and \( P^2 \)). 
This proves that $b$ is complete or incomplete to $N^-(v)$.

Now, if $b$ is non-adjacent to vertices $v_1\in N^-(v)$ and $v_2\in N^+(v)$. 
Then, $G$ induces a claw on $\{v,b,v_1,v_2\}$, a contradiction. This proves (ii).
\end{proof}

\begin{lem}\label{lem:box2}
	Let \( v \) be a vertex  in \( G_0 \).
	If there exists a simple path from \( v \) to \( u_1 \) in \( G_0 \) which passes through a non-singular component \( C \), then \( v \) is incomplete to \( B_0 \).
\end{lem}

\begin{proof}
	For the contrary, suppose that $v$ is adjacent to a vertex $b\in B_0$ and there is a simple path $P$ from $v$ to $u_1$ which passes through a non-singular level component $C$ in $L_i$, for some $i\geq 3$. 
Let $P^0$ be an induced path of length at most two from $u_1$ to $b$ (i.e. either $P^0=u_1-b$ or $P^0=u_1-u_0-b$).
	We consider two possibilities:
	\paragraph{Case 1.} $v\in V(C)$. 
Since $C$ is non-singular, there is an induced path from $C$ to $\eta(C)$ whose inner vertices are in $L_{>i}$. 
First, suppose that there is such a path $P'$ which starts from $v$ and ends in a vertex $w$ in $\eta(C)$. 
Let $P''$ be a simple path from $w$ to $u_1$.
Let $v'\in L_{i+1}$ and $v''\in L_{i-1}$ be the second vertices in $P'$ and $P$ respectively.
By Lemma~\ref{lem:box}, $b$ is adjacent to either $v'$ or $v''$ and not both. 
In the former case, the holes $b-v'-P'-w-P''-u_1-P^0-b$ and $b-v-P-u_1-P^0-b$ have different parities, a contradiction. 
In the latter case, the holes $b-v''-P-u_1-P^0-b$ and $b-v-P'-w-P''-u_1-P^0-b$ have different parities, a contradiction.
	
Now, suppose that there is no such path $P'$ from $v$ to $\eta(C)$. 
By the definition, there is a path $P'$ from a vertex $u\neq v$ in $V(C)$ to a vertex $w$ in $\eta(C)$ whose inner vertices lie in $L_{>i}$.
	Also, let $v',v''$ and $P''$ are defined as above and let $w''\in L_{i-1}$ be the second vertex in $P''$. 
By the assumption, $v$ is non-adjacent to $v'$. 
Also, $b$ is non-adjacent to $u$ (otherwise, we can replace the role of $v$ with $u$ and the above argument yields a contradiction). 
Also, by Lemma~\ref{lem:H_inner_level_king vertex}, $u$ is a king vertex of $C$ and so $u$ is adjacent to $v$. 
Note that $v''$ is adjacent to $u$, since otherwise the path $v''-v-u-P'-w-w''$ is a forbidden path, a contradiction with Lemma~\ref{lem:forbidden_path}.
	Now, we prove that $b$ is non-adjacent to $v''$. 
To see this, suppose that $b$ is adjacent to $v''$. 
Then, if $v''\in L_{\geq 3}$, then by Lemma~\ref{lem:box}, $b$ is adjacent to $u$, a contradiction and if $v''\in L_2$, then $G$ induces a claw on $\{v'',w'',u,b\}$ (since $L_2$ is a clique).
This shows that $b$ is non-adjacent to $v''$. 
Now, the holes $b-v-P-u_1-P^0-b$ and $b-v-u-P'-w-P''-u_1-P^0-b$ have different parities, a contradiction. 
This yields that Case 1 cannot happen.
	
	\paragraph{Case 2.} $v\notin V(C)$. 
Without loss of generality, we can assume that $b$ is non-adjacent to inner vertices of \( P \). 
Also, we can assume that $C$ is the first non-singular level component that \( P \) visits. 
Let $u$ be the vertex in $V(P)\cap V(C)$. 
First, assume that there is an induced path $P'$ from $u$ to a vertex $w$ in $\eta(C)$ whose inner vertices are in $L_{>i}$. 
Also, let $u'$ be the vertex before $u$ in the path $P$ and $w'$ be the vertex after $u$ in the path $P'$. 
Then, $u',w'$ are in the same level component $C'$ which is singular (note that non-singularity of $ C'$ is in contradiction with the choice of $C$). 
Thus, $P'=u-w'-w$.  
	By Case 1, $b$ is non-adjacent to $w$ and $u$. 
If $u'\neq v$, define $u''$ be the vertex before $u'$ in $P$ and if $u'=v$, define $u''=b$. 
Thus, $u'$ is non-adjacent to $w$, since otherwise $G$ induces a claw on $\{u',u,w,u''\}$. 
This implies that $u'$ and $w'$ are distinct. 
Also, $w'$ is non-adjacent to $u''$, since otherwise $G$ induces a claw on $\{w',u,w,u''\}$. 
Now, we have two holes $b-v-P-u_1-P^0-b$ and $b-v-P-u'-w'-w-P''-u_1-P^0-b$ whose lengths differ by one, a contradiction.
	
Now, suppose that there is no such path $P'$ from $u$ to $\eta(C)$. 
By the definition, there is an induced path $P'$ from a vertex $z\neq u$ in $V(C)$ to a vertex $w$ in $\eta(C)$. 
Let $u'\in L_{i+1}$ be the vertex before $u$ in the path $P$ and $z'\in L_{i+1}$ be the vertex after $u'$ in the path $P'$. 
Then, we can assume that $u$ is non-adjacent to $z'$ and $u'$ is non-adjacent to $z$ since otherwise, we have a situation discussed above. 
Thus, by Corollary~\ref{col:xy_edge}, \( u' \) is non-adjacent to \( z' \).
Also, by Lemma~\ref{lem:H_inner_level_king vertex}, \( z' \) is a king vertex of its level component, so \( u' \) and \( z' \) are located in different level components.
Hence, by choice of $C$, every inner vertex of $P'$ and every vertex of $P$ are in different level components. 
Therefore, we have two holes $b-v-P-u_1-P^0-b$ and $b-v-P-u-z-P'-w-P''-u_1-P^0-b$ which are of different parities, a contradiction. 
This proves that Case 2 cannot happen.	
\end{proof}

Now, we slightly modify  Algorithm~\ref{alg:main} to prove that two cops are sufficient to capture the robber in all (claw, even-hole)-free graphs.
To do this, in the initialization round, we locate the cops $c_1$ and \( c_2 \) on an arbitrary vertex $u_0$ and if $P^0=u_0-u_1-\cdots$ is a path from $u_0$ to the robber's position, then we move $c_1$ to $u_1$. 
Now, if $c_1$ is in a level component $C_1$, as long as $C_1^+$ is singular, we move $c_1$ to a king vertex of $C_1^+$. 
Once $C_1^+$ is non-singular, $c_2$ will be free and we perform the synchronizing step, i.e. we move $c_2$ to the level component $C_1$. 
Then, we act exactly the same as Algorithm~\ref{alg:main}. 
By the argument in Subsection~\ref{sub:correct}, during the execution of the algorithm, if the cop $c_1$ is in a level $L_i$, then the robber cannot move to a lower level than $L_i$ using the vertices of $G_0$. 
We just need to show that the robber cannot move to the vertices in $B_0$. 
To see this, note that as long as $C_1^+$ is singular, there is a cop on $u_0$, so the robber cannot move to $B_0$ unless he is captured. 
Suppose that, for the first time, $C_1^+$ is non-singular. 
Afterwards, if the robber is in a vertex $v\in V(G_0)$, then by Lemma~\ref{lem:correctness}, every simple path from $v$ to $u_1$ passes through the level components $C_1^+$ or $\eta(C_1^+)$ which are non-singular. 
So, by Lemma~\ref{lem:box2}, $v$ is incomplete to $B_0$ and so the robber cannot reach $B_0$. 
This proves that the robber is trapped in $G_0$ and will be captured by $c_1$ or $c_2$.
Finally, note that in the first synchronizing step, $c_1$ might move on the already visited vertices, so the capture time cannot be bounded by $n$; however, it is still a most  $2n$.

\section{Concluding Remarks}
In this paper, we proved that the cop number of all (claw, even-hole)-free graphs is at most two. 
This can be considered as a first step towards answering Question~\ref{qst:main} asking that if all even-hole-free graphs are cop bounded. 
Since the cop number is bounded by treewidth and it is shown that (triangle, even-hole)-free graphs have bounded treewidth \cite{CAMERON2018463}, they are another subclasses of even-hole-free graphs which are cop-bounded.

As Sintiari et al. \cite{Sintiari:2019aa} proved that ($K_4$, even-hole)-free graphs have unbounded treewidth, a second approach to answering Question~\ref{qst:main} could be investigating the cop number of ($K_4$, even-hole)-free graphs. 
It is noteworthy that a subclass of ($K_4$, even-hole)-free graphs with unbounded treewidth introduced in \cite{Sintiari:2019aa}, called ``layered wheels'', happens to be cop-bounded. 
Moreover, Cameron et al. \cite{cameron2018structure} studied the structure of the class of (pan, even-hole)-free graphs, which is a superclass of (claw, even-hole)-free graphs. 
So, it is natural to ask how one can generalize our results to (pan, even-hole)-free graphs.


\newcommand{\etalchar}[1]{$^{#1}$}

\end{document}